\documentclass[aap,MSNbibl,citesort,dvips]{arximspdf}
\usepackage{graphicx}

\doi{10.1214/11-AAP811} 
\volume{22}
\issue{4}
\pubyear{2012}
\firstpage{1728}
\lastpage{1743}

\makeatletter
\newtheorem{duge}{Lemma}[section]
\newproclaim{rem}[duge]{Remark}
\newtheorem{prop}[duge]{Proposition}
\newtheorem{theo}[duge]{Theorem}
\newtheorem{cor}[duge]{Corollary}
\renewcommand{\P}{\mathbf{P}}
\newcommand{\E}{\mathbf{E}}
\newcommand{\Z}{\mathbb{Z}}
\newcommand{\R}{\mathbb{R}}
\newcommand{\T}{\mathbb{T}}
\newcommand{\un}{\mathbf{1}}
\newcommand{\defeq}{\stackrel{\mathrm{def}}{=}}
\newcommand{\pere}[1]{\stackrel{\leftarrow}{#1}}
\newcommand{\fils}[2]{\stackrel{\rightarrow}{#1}^{{#2}}}
\makeatother

\begin{document}
\begin{frontmatter}

\title{Continuous-time vertex reinforced jump processes on Galton--Watson trees}
\runtitle{Vertex reinforced jump processes}

\begin{aug}
\author[A]{\fnms{Anne-Laure} \snm{Basdevant}\corref{}\ead[label=e1]{anne.laure.basdevant@normalesup.org}}
\and
\author[B]{\fnms{Arvind} \snm{Singh}\ead[label=e2]{arvind.singh@math.u-psud.fr}}
\runauthor{A.-L. Basdevant and A. Singh}
\affiliation{Universit\'{e} Paris Ouest and Universit\'{e} Paris Sud}
\address[A]{Laboratoire Modal'X\\
Universit\'{e} Paris Ouest \\
92000 Nanterre\\
France\\
\printead{e1}} 
\address[B]{D\'{e}partement de Math\'{e}matiques\\
Universit\'{e} Paris-Sud\\
91405 Orsay Cedex\\
France\\
\printead{e2}}
\end{aug}

\received{\smonth{6} \syear{2010}}
\revised{\smonth{9} \syear{2011}}

%
\begin{abstract}
We consider a continuous-time vertex reinforced jump process on a
supercritical Galton--Watson tree. This process takes values in the
set of vertices of the tree and jumps to a neighboring vertex with
rate proportional to the local time at that vertex plus a constant~$c$.
The walk is either transient or recurrent depending on this
parameter $c$. In this paper, we complete results previously
obtained by Davis and Volkov
[\textit{Probab. Theory Related Fields} \textbf{123} (2002) 281--300,
\textit{Probab. Theory Related Fields} \textbf{128} (2004) 42--62]
and Collevecchio
[\textit{Ann. Probab.} \textbf{34} (2006) 870--878,
\textit{Electron. J. Probab.} \textbf{14} (2009) 1936--1962]
by proving that there is a unique (explicit)
positive $c_{\mathrm{crit}}$ such that the walk is recurrent
for $c \leq c_{\mathrm{crit}}$ and transient for $c >
c_{\mathrm{crit}}$.
\end{abstract}

%


\end{frontmatter}

\section{Introduction}\label{sectionintro}

The model of the continuous-time vertex reinforced jump process
(VRJP) introduced by Davis and Volkov \cite{DavisVolkov02} may be
described in the following way: let $G$ be a locally finite graph
and pick $c>0$. Call $\operatorname{VRJP}(c)$ a continuous-time process
$(X(t),t\geq0)$ on the vertices of $G$, starting at time $0$ at
some vertex $v_0\in G$ and such that, if $X$ is at a vertex $v\in G$
at time $t$, then, conditionally on $(X(s),s\le t)$, the process $X$
jumps to a~neighbor $u$ of $v$ with rate
%
\begin{equation}\label{defLocalTimeL}
L_c(t,u) \defeq c+\int_{0}^t\un_{\{X(s)=u\}}\,ds.
\end{equation}
Equivalently, the walk stays at site $v$ an exponential time of
parameter $\sum_{u\sim v} L_c(t, u)$ and then jumps to a neighbor
$u$ with a probability proportional to $L_c(t,u)$.

The case $G = \Z$ was investigated by Davis and Volkov
\cite{DavisVolkov02} who proved that, for any $c>0$, the $\operatorname{VRJP}(c)$
is recurrent and the proportion of time spent at each site converges
jointly to some nondegenerate distribution. In a~subsequent article~\cite{DavisVolkov04},
the same authors studied the VRJP on more
general graphs. They showed that when $G$ is a tree, the walk can
either be recurrent or transient. For a regular $b$-ary tree (more
generally, a~tree satisfying a~so-called $L$-property), they proved
the existence of two constants
%
\begin{equation}\label{refctcr}
0 < c_r(b)\leq c_t(b)
\end{equation}
such that:
\begin{itemize}
\item For $c< c_r$, the $\operatorname{VRJP}(c)$ visits every vertex infinitely often a.s.
\item For $c> c_t$, the $\operatorname{VRJP}(c)$ visits every vertex only a finite
number of time a.s.
\end{itemize}
Although they did not prove that $c_r = c_t$, the computation of the
bound $c_t$ obtained in \cite{DavisVolkov04} already implies that
the $\operatorname{VRJP}(1)$ is transient on a 4-ary tree. More recently,
Collevecchio \cite{CollevecchioO6,Collevecchio09} showed that the
$\operatorname{VRJP}(1)$ on a $3$-ary tree is also transient with positive speed (and
a C.L.T. holds) and asked whether this result also holds for a
$\operatorname{VRJP}(1)$ on a binary tree.

The main result of this paper states that, for almost every
realization of an infinite supercritical Galton--Watson tree with
mean offspring distribution~$b$, one has $c_t(b) = c_r(b)$ and
recurrence occurs at the critical value. In fact, recalling
Lyons--Pemantle's criterion for recurrence/transience of a~random
walk in random environment (RWRE) on a Galton--Watson tree (see
Theorem 3 of \cite{LyonsPemantle92}), Theorem~\ref{maintheo} states
that the phase transition of a~$\operatorname{VRJP}(c)$ is exactly the same as
that of a discrete-time random walk in an i.i.d. random environment
where the law of the environment is given by the random variable
$m_c(\infty)$ defined below.

Concerning the discrete-time model of the linearly edge reinforced
random walk (LERRW), de Finetti's theorem implies that any LERRW on
an acyclic graph may be seen as a RWRE in a Dirichlet environment.
However, the non-exchangeability of the increments of a VRJP forbids
a direct interpretation of the process in terms of a time change of
a RWRE and we do not have a~convincing argument why the VRJP should
have the same phase transition as a RWRE (see Davis and Dean
\cite{DavisDean11} for a study of the relations between these models
in the one-dimensional case). For example, using Theorem~1.5 of
\cite{Aidekon08}, one can check that, on a regular tree, the random
walk in the random environment defined by $m_c(\infty)$ always has a
positive speed when it is transient. Does this result somehow imply
that a transient VRJP always has positive speed?

\begin{theo}\label{maintheo} For $c>0$, let $m_c(\infty)$ denote a
random variable on $(0,\infty)$ with density
%
\begin{equation}\label{theoform1}
\P\{m_c(\infty) \in dx\} \defeq\frac{c\exp(-
{(c(x-1))^2}/{2x})}{\sqrt{2\pi x^3}}\,dx.
\end{equation}
Define
%
\begin{equation} \label{theoform2}
\mu(c) \defeq\inf_{a\in\R}\E[m_c(\infty)^a] =
\frac{c}{\sqrt{2\pi}}\int_0^\infty
x^{-1}\exp\biggl(-\frac{(c(x-1))^2}{2x}\biggr)\,dx.
\end{equation}
Let $\T$ denote a Galton--Watson tree with mean $1 < b < \infty$. On
the event that~$\T$ is infinite, we have, for almost every realization
of $\T$:
\begin{itemize}
\item If $b\mu(c) \le1$, the $\operatorname{VRJP}(c)$ on $\T$ visits every vertex
infinitely often a.s.
\item If $b\mu(c) > 1$, the $\operatorname{VRJP}(c)$ on $\T$ visits every vertex only
finitely many times a.s.
\end{itemize}
\end{theo}

For $c=1$, we have $1/\mu(1)\simeq1.095$. Therefore the $\operatorname{VRJP}(1)$ is
transient on any regular $b$-ary tree with $b\geq2$. Making a
change of variable (see Appendix of \cite{DavisVolkov04}), the
function $\mu$ may be rewritten in the form
\[
\mu(c)=\frac{1}{\sqrt{2\pi}}\int_{-\infty}^{\infty}\frac
{e^{-y^2/2}}{\sqrt{1+y^2/(4c^2)}}\,dy.
\]
Thus, $\mu$ is continuous, strictly increasing on $[0,\infty)$ with
$\lim_0 \mu= 0$ and \mbox{$\lim_\infty\mu= 1$} (see Figure \ref{fignu}).
Denoting by $\mu^{-1}$ its inverse, we get the following.

\begin{cor} For any supercritical Galton--Watson tree with mean
$1<b<\infty$, with the notation (\ref{refctcr}), we have, for almost
every realization where the
tree is infinite,
\[
c_t(b) = c_r(b) = \mu^{-1}(1/b).
\]
In particular, the recurrence/transience phase transition for VRJP
on the class of Galton--Watson tree is monotonic w.r.t. the
reinforcement parameter $c$; that is, if the $\operatorname{VRJP}(c)$ is
transient for some $c>0$, then the $\operatorname{VRJP}(\tilde{c})$ is transient
for any $\tilde{c} \geq c$.
\end{cor}

Let us note that, although this monotonicity result w.r.t. the
parameter~$c$ seems quite natural, we do not know how to prove it
without using the explicit computation of $\mu$ to assert that this
function is monotonic. More generally, we do not know how to prove
a similar result for an infinite graph which contains loops.

\begin{figure}

\includegraphics{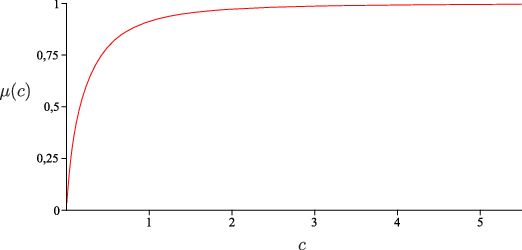}

\caption{Graph of the function $\mu$.}\label{fignu}
\end{figure}

%

\section{Preliminary results}\label{sectionpreliminaries}
In this section, we recall some important results concerning VRJP
obtained by Davis and Volkov in \cite{DavisVolkov02,DavisVolkov04}
which will play a~key role in the proof of Theorem \ref{maintheo}.
We start with the so-called \textit{restriction principle} for VRJP
which follows from the lack of memory of the exponential law.
\begin{prop}[(Restriction principle; Davis, Volkov \cite{DavisVolkov04})]
Let $G$ be a connected graph and let $G_1$ be a connected subgraph
with the property that for any path starting in any $v\in
G\setminus G_1$ and ending in $G_1$, the first ``port of entry'' into\vadjust{\goodbreak}
$G_1$ is uniquely determined. Assume moreover that on each connected
component of $G\setminus G_1$, the $\operatorname{VRJP}(c)$ is recurrent. Then the
$\operatorname{VRJP}(c)$ on $G$ starting at $v\in G_1$ restricted to $G_1$ has the
same law as the $\operatorname{VRJP}(c)$ on the subgraph $G_1$ starting from the
same point.
\end{prop}

We shall make intensive use of this result in the case where $G$ is
a rooted tree and $G_1$ is a subtree of $G$ (e.g., the ball
of radius $N$ centered at the~root).

\subsection{VRJP on the graph $\{0,1\}$}\label{sectionprelim01}

In view of the restriction principle stated above, many properties
of the VRJP on an acyclic graph can be derived from the study of the
VRJP on the simpler graph $G_0 \defeq\{0,1\}$. A detailed analysis
of the VRJP on $G_0$ is undertaken in \cite{DavisVolkov02}. Consider
a $\operatorname{VRJP}(c)$ on $G_0$, starting at $0$. For $t \geq c$, define the
stopping time
\[
\xi(t)\defeq\inf\{s>0, L_c(s,0)=t\}
\]
and
%
\begin{equation}\label{defAc}
A_c(t)\defeq L_c(\xi(t),1).
\end{equation}
The quantity $A_c(t)-c$ corresponds to the time spent at site $1$
before spending time $t-c$ at site $0$. The variable $A_c(t)$ takes
values in $[c,\infty)$ and has an atom at $c$. More precisely,
denoting by $\mathcal{E}(c)$ an exponential random variable with
parameter $c$, we have
%
\begin{eqnarray}\label{tes}
\P\{A_c(t)=c\} &=& \P\{\mbox{the $\operatorname{VRJP}(c)$ does not jump
before time $t-c$}\}\nonumber\\
&=&\P\{\mathcal{E}(c)>t-c\}\\
&=&e^{-c(t-c)}.\nonumber
\end{eqnarray}
For $t>c$, the law of $A_c(t)$ conditioned on $\{A_c(t)>c\}$ is
absolutely continuous w.r.t. the Lebesgue measure, with strictly
positive density on $(c,\infty)$. Considering only the time spent at
site $1$ before the first return to site $0$, we get the lower
bound:
%
\begin{equation}\label{densitepositive}
\P\{A_c(t)\geq\alpha  |  A_c(t)> c\}\geq
\P\{\mathcal{E}(t)>\alpha-c\} = e^{-(\alpha-c)t}.\vadjust{\goodbreak}
\end{equation}
For $t\geq c$, define
\[
m_c(t)\defeq\frac{A_c(t)}{t}.
\]
It is proved in \cite{DavisVolkov02} that the process $(m_c(t),t\geq
c)$ is a positive martingale which converges a.s. toward the random
variable $m_c(\infty)$ defined in Theorem \ref{maintheo}. The
moments of $m_c(\infty)$ can be computed explicitly using
(\ref{theoform1}). For $\theta\in\R$, we get
\[
\E[m_c(\infty)^\theta]=\sqrt{\frac{2}{\pi}}ce^{c^2}K_{\theta-1/2}(c^2)
<\infty,
\]
where $K_\alpha(x)$ denotes the modified Bessel function of the
second kind of order~$\alpha$ (cf. \cite{AbramowitzStegun64}
for details on this class of special functions). Using $K_{\alpha} =
K_{-\alpha}$ and $K_{\alpha} \leq K_{\alpha'}$ for $0\leq\alpha
\leq\alpha'$, it follows that
%
\begin{equation}\label{mintminfty}
\min_{\theta\in\R}
\E[m_c(\infty)^\theta]=\E\bigl[\sqrt{m_c(\infty)}\bigr],
\end{equation}
which entails the second equality of (\ref{theoform2}).

\subsection{VRJP on trees}\label{sectionprelimTree}

Let $T$ be a deterministic locally bounded tree rooted at some
vertex $o$. According to Theorem $3$ of \cite{DavisVolkov04}, any
VRJP on $T$ is either recurrent (every vertex is visited infinitely
often a.s.) or transient (every vertex is visited only finitely many
times a.s.). Moreover, we have the following characterization of
recurrence and transience in terms of the local time of the walk at
the root:
%
\begin{equation}\label{equivalencerecurrence}
\mbox{The $\operatorname{VRJP}(c)$ on $T$ is recurrent} \quad\Longleftrightarrow\quad
\lim_{t\rightarrow\infty} L_c(t,o)=\infty.
\end{equation}
Define, for $t>c$,
\[
\xi(t)\defeq\inf\{s>0, L_c(s,o)=t\},
\]
and let $(v_0=o,v_1,\ldots,v_n)$ be a nearest-neighbor
self-avoiding path starting from the root of $T$ and ending at
$v_n$. For $0\le k \le n$, set
%
\begin{equation}\label{evolutionbranche}
Z_k\defeq L_c(\xi(t),v_k).
\end{equation}
If $T$ is a finite tree, then the $\operatorname{VRJP}(c)$ on $T$ is recurrent.
Applying the restriction principle to the subgraph
$(v_0=o,v_1,\ldots,v_n)$, it follows that the process $(Z_k)_{0\le k
\le n}$ is a Markov chain starting from $Z_0 = t$ with transition
probabilities
%
\begin{equation}\label{transitionProbaZ}
\P\{ Z_{k+1} \in E  |  Z_0,\ldots,Z_k = x\} = \P\{ A_c(x) \in E\},
\end{equation}
where $A_c$ is the random variable defined in (\ref{defAc}). Let us
note that $Z$ takes values in $[c,\infty)$ and that $c$ is an
absorbing point. Moreover, since $(A_c(t)/t)_{t\geq c}$ is a
martingale starting from $1$, the process $Z$ is also a (positive)
martingale. Therefore, $Z_n$ converges a.s. as $n$ tend to infinity
and the limit is necessarily equal to $c$ a.s.

\section{\texorpdfstring{Proof of Theorem \protect\ref{maintheo}}{Proof of Theorem 1.1}}\label{sectionProof}

We first set some notation. Let $\mathcal{T}$ be the set of all
locally finite rooted trees. Given a tree $T\in\mathcal{T}$, we
denote its root by $o$. For $v\in T$, we use the notation $\pere{v}$
for the father of $v$ and $\fils{v}{1},\fils{v}{2},\ldots$ for the
sons of $v$. We also denote by $|v|$ the height of the vertex $v$ in
the tree (i.e., its graph distance from the root). For $n\geq
0$, $T_n$ will stand for the subtree of $T$ of vertices of height
smaller than or equal to $n$.

In the following, $\nu$ will always denote a probability measure on
the nonnegative integers with finite mean $b>1$ and
$\mathbb{Q}_{\nu}$ will denote the probability measure on
$\mathcal{T}$ under which the canonical r.v. $\T$ is a Galton--Watson
tree with offspring distribution $\nu$.

For $c>0$, we consider on the same (possibly enlarged) probability
space a~process $X = (X(t),t\geq0)$ and a collection of probability
measures $(P_{T,c},\allowbreak  T\in\mathcal{T})$ called quenched laws such
that $X$ under $P_{T,c}$ is a $\operatorname{VRJP}(c)$ on $T$ with $X(0)=o$. The
annealed probability is defined by
\[
\mathbb{P}_{\nu,c} \defeq P_{\T,c} \otimes\mathbb{Q}_\nu.
\]
We say that $X$ under $\mathbb{P}_{\nu,c}$ is a $\operatorname{VRJP}(c)$ on a
Galton--Watson tree with reproduction law $\nu$. In the following, we
shall omit the subscripts $c,\nu$ when it does not lead to
confusion.\vspace*{-2pt}

\subsection{Restriction to trees without leaves}\label{sectionRestrictNoLeaf}
The Harris decomposition of a~supercritical Galton--Watson tree
states that conditionally on non-extinc\-tion,~$\T$ under
$\mathbb{Q}_\nu$ can be generated in the following way:
\begin{itemize}
\item Generate a Galton--Watson tree $\T_g$ with no leaf called the backbone.
\item Attach at each vertex $v$ of $\T_g$ a random number $N_v$ of
i.i.d. subcritical trees $\T_l^1,\ldots,\T_l^{N_v}$.
\end{itemize}
See, for instance, \cite{AthreyaNey72} for a precise description of
the laws of $N_v$, $\T_g$ and $\T_l$. Let us simply note that the
expected number of children per vertex of $\T_g$ is also equal to
$b$. Consider now a $\operatorname{VRJP}(c)$ on $\T$ on the event that $\T$ is
infinite. The restriction principle applied with $G=\T$ and
$G_1=\T_g$ implies that the $\operatorname{VRJP}(c)$ on $\T$ is transient if and
only if the $\operatorname{VRJP}(c)$ on $\T_g$ is transient. Since the criterion
for the transience/recurrence of the walk of Theorem \ref{maintheo}
only depends on $b$, it suffices to prove the result for trees
without leaves. In the sequel, we will always assume that this is
the case, that is,
\[
\nu(0)=0.\vspace*{-2pt}
\]

\subsection{\texorpdfstring{Proof of recurrence when $b\mu(c)<1$}{Proof of recurrence when b mu(c)<1}}\label{sectionrecurrence}
In \cite{DavisVolkov04}, Davis and Volkov proved that a $\operatorname{VRJP}(1)$ is
recurrent when $b\leq1.04$. In fact, their argument shows
recurrence whenever $b\mu(c)<1$ by simply fine-tuning some
parameters. We provide below a sketch of the proof and we refer the
reader to \cite{DavisVolkov04} for further details.

Consider a $\operatorname{VRJP}(1)$ on the nonnegative integers $\{0,1,\ldots\}$ and
denote by~$\sigma_n$ the first time the walk reaches level $n$.
It\vadjust{\goodbreak}
is proved in the Appendix of~\cite{DavisVolkov04} that, for any
$a>1$,
%
\begin{equation}\label{oldB}
\P\{L_1(\sigma_n,0)<a^n\}\le\bigl(\E\bigl[\sqrt{m_1(\infty)}\bigr]a^{1/2}\bigr)^n.
\end{equation}
Adapting the proof for any $c>0$, it is immediate to check that, for
any $\operatorname{VRJP}(c)$ on the nonnegative integers,
%
\begin{equation}\label{newB}
\P\{L_c(\sigma_n,0)<a^n\}\le
\bigl(\E\bigl[\sqrt{m_c(\infty)}\bigr]a^{1/2}\bigr)^n=(\mu(c)a^{1/2})^n.
\end{equation}
We now copy the argument of the proof of Theorem 5 of
\cite{DavisVolkov04} using the bound~(\ref{newB}) in place of
(\ref{oldB}). Let $T\in\mathcal{T}$ be an infinite tree and let $X$
denote a~$\operatorname{VRJP}(c)$ on $T$. Let $V_n$ denote the number of vertices
of $T$ of height $n$ and set
\[
G_n=L_c\bigl(\inf\{t>0,|X(t)|=n\},o\bigr)
\]
so that $G_n - c$ is the total time spent by $X$ at the root before
reaching a~vertex of height $n$. Conditioning on the position of $X$
when it reaches level~$n$ and applying the restriction principle to
the path connecting this vertex to the root, we find, using~(\ref{newB}),
%
\begin{equation}\label{borneGn}
P_T\{G_n<a^n\}\leq(\mu(c)a^{1/2})^n V_n.
\end{equation}
Assume now that the tree $T$ satisfies
\[
\liminf_{n\rightarrow\infty} V_n^{1/n}< \mu(c)^{-1};
\]
then (\ref{borneGn}) yields, taking $a$ sufficiently close to 1,
\[
P_T\{G_{n_k}<a^{n_k}\}\leq(1-\varepsilon)^{n_k}
\]
for some subsequence $(n_k)$ and some $\varepsilon>0$. Letting $k$
go to infinity, we get that
\[
\lim_{t\rightarrow\infty} L_c(t,o)=\infty \qquad\mbox{$P_T$-a.s.}
\]
Thus, the $\operatorname{VRJP}(c)$ on $T$ is recurrent according to
(\ref{equivalencerecurrence}). We conclude the proof for the
$\operatorname{VRJP}(c)$ on the Galton--Watson tree $\T$ noticing that, when
$b\mu(c) < 1$, we have for $\mathbb{Q}_\nu$-almost any tree $T \in
\mathcal{T}$,
\[
\lim_{n\rightarrow\infty} V_n^{1/n} = b < \mu(c)^{-1}.
\]

\subsection{The branching Markov chain $F$}\label{sectionBMC}

Recall that we assume $\nu(0) = 0$ so the tree $\T$ is infinite
$\mathbb{Q}_{\nu}$-a.s. We introduce a branching Markov chain $F$
indexed by the vertices of $\T$ and taking values in $[c,\infty)$,
\[
F \defeq\bigl(f(v), v\in\T\bigr) \in\bigcup_{T\in\mathcal{T}}[c,\infty]^T.
\]
More precisely, the population at time $n$ is indexed by
$\{v\in\T,  |v|=n\}$ and the set of positions of the particles of
$F$ at time $n$ is
\[
F_n \defeq\bigl(f(v),  |v|=n\bigr).\vadjust{\goodbreak}
\]
Thus, the genealogy of this branching Markov chain is chosen to be
exactly the Galton--Watson tree $\T$. In particular, under the
annealed probability~$\mathbb{P}$, each particle $v$ splits, after a
unit of time, into a random number $B$ of particles
$\fils{v}{1},\ldots,\fils{v}{B}$ where $B$ is distributed as $\nu$.
In order to characterize $F$, it remains to specify the law of the
position $f(v)$ of the particles. We choose the dynamics of $F$,
conditionally on its genealogy $\T$ in the following way:
\begin{longlist}[(a)]
\item[(a)] For any $n >0$, conditionally on $(f(u),  |u|< n)$, the
random variables $(f(v),|v|=n)$ are independent.
\item[(b)] For any $v \neq o$, conditionally on $(f(u),  |u|< |v|)$,
the random variable~$f(v)$ is distributed as $A_c(f(\pere{v}))$ where
$A_c$ is
defined by (\ref{defAc}).
\end{longlist}
We use the notation $\mathbb{P}_{x_0}$ for the annealed law where
$F$ starts with the initial particle~$o$ being located at $f(o) =
x_0$. Note that, since the tree is infinite, the Markov chain $F$
never becomes extinct. However, recalling that $c$ is an absorbing
point for~$A_c$, it follows that if a particle $v$ is located at
$f(v) = c$, then all its descendants are also located at $c$. Thus,
we will say that the process $F$ \textit{dies out} if there exists a
time $n$ such that all the particles at time~$n$ are at position $c$.
Otherwise, we say that the process \textit{survives}.

\begin{prop}\label{propmonotonyF}
For any $x \leq y$, the process $F$ under $\mathbb{P}_{x}$ is
stochastically dominated by $F$ under $\mathbb{P}_{y}$.
\end{prop}

\begin{pf}
Recalling (\ref{defAc}), it is clear that $A_c(x) \leq A_c(y)$ for
any $c\leq x\leq y$ and the result follows by induction.
\end{pf}

\begin{prop}
Let $x_0>0$ and $N>0$ and let $(X^N(t),t\geq0)$ denote a $\operatorname{VRJP}(c)$
on the finite subtree $\T_N = \{v\in\T,  |v|\leq N\}$, with $X^N(0)
= o$. Set
\[
\xi^N(x_0)\defeq\inf\{s>0, L^N_c(s,o)=x_0\},
\]
where $L^N$ is defined as in \textit{($\ref{defLocalTimeL}$)} for
$X^N$. Then, the collections of random variables
$(L^N_c(\xi^N(x_0),v),  v \in\T_N)$ under $\mathbb{P}$ and
$(f(v),  v \in\T_N)$ under $\mathbb{P}_{x_0}$ have the same law.
\end{prop}

\begin{pf}
Simply notice that since the $\T_N$ is finite, $X^N$ is recurrent
and~$\xi^N$ is finite a.s. and apply the restriction principle for
VRJP.
\end{pf}

The VRJPs $X$ on $\T$ and $X^N$ on $\T_N$ coincide up to the first
time they reach a site of height $N$; therefore,
\begin{eqnarray*}
&&\mathbb{P}\{\mbox{$X$ reaches level $N$ before spending time $x_0-c$ at
the origin}\} \\
&&\qquad= \mathbb{P}\{\mbox{$X^N$ reaches level $N$ before spending time
$x_0-c$ at the origin}\} \\
&&\qquad= \mathbb{P}_{x_0}\{\mbox{the process $F$ does not die out before
time $N$}\}.
\end{eqnarray*}
Letting $N$ and then $x_0$ tend to infinity, and using
(\ref{equivalencerecurrence}), we get
\begin{eqnarray}\label{limitFsurvive1}
&&\mathbb{P}\{\mbox{$X$ visits every vertex of $\T$ finitely many
times}\}
\nonumber
\\[-8pt]
\\[-8pt]
\nonumber
&&\qquad=
\mathop{\lim  \uparrow}_{x_0\to\infty}\mathbb{P}_{x_0}\{\mbox{$F$
survives}\}.
\end{eqnarray}
The next proposition extends the $0-1$ law proved in
\cite{DavisVolkov04} for deterministic trees to Galton--Watson trees.

\begin{prop}[($0-1$ law for VRJP on Galton--Watson
trees)]\label{prop0-1} Let $\T$ be a Galton--Watson tree $\T$ without
leaves and with mean
$b>1$. Then, for any $c>0$, the $\operatorname{VRJP}(c)$ $X$ on $\T$ is either
recurrent or transient under the annealed law:
\begin{eqnarray*}
&&\mathbb{P}\{\mbox{$X$ visits every vertex of $\T$ finitely many
times}\} \\
&&\qquad= 1 - \mathbb{P}\{\mbox{$X$ visits every vertex of $\T$
infinitely often}\} \in\{0,1\}.
\end{eqnarray*}
\end{prop}

\begin{pf} Since the $0-1$ law holds for any deterministic tree, we
just need to show that the r.h.s. limit of (\ref{limitFsurvive1}) is
either $0$ or $1$.
Suppose that this limit is nonzero. We can find $x_0 >c$ and
$\alpha> 0$ such that
\[
\mathbb{P}_{x_0}\{\mbox{$F$ survives}\} \geq\alpha.
\]
Given an interval I, let $N^{I}_k$ denote the number of particles in
$F$ located inside~$I$ at time $k$, that is,
%
\begin{equation}\label{defNI}
N^{I}_k \defeq\#\{v\in\T, |v|=k \mbox{ and } f(v) \in I\}.
\end{equation}
Since the particles in $F$ evolve independently, conditionally on
\mbox{$(f(v), |v|\leq k)$}, the process $(f(v), |v|\geq k)$ has the same
law as the union of $\#\{v\in\T,  |v|=k\}$ independent branching
Markov chains $F$ starting from the positions $F_k = (f(v),
|v|=k)$. Making use of the stochastic monotonicity of $F$ w.r.t. the
position of the initial particle (Proposition
\ref{propmonotonyF}), we deduce that, for any $\varepsilon>0$, we
can find $m$ large enough such that, for any $k$ and any $x$,
%
\begin{eqnarray}\label{inter0}
\mathbb{P}_{x}\{\mbox{$F$ survives}\} &\geq& \mathbb
{P}_{x}\bigl\{\mbox{$N_k^{[x_0,\infty)} \geq m$ and $F$ survives}\bigr\}\nonumber\\
&\geq& \mathbb{P}_{x}\bigl\{\mbox{$N_k^{[x_0,\infty)} \geq
m$}\bigr\}
(1 -\mathbb{P}_{x_0}\{\mbox{$F$ dies out}\}^m)
\nonumber
\\[-8pt]
\\[-8pt]
\nonumber
&\geq& \mathbb{P}_{x}\bigl\{\mbox{$N_k^{[x_0,\infty)} \geq
m$}\bigr\}
\bigl(1 - (1-\alpha)^m\bigr)\\
&\geq& \mathbb{P}_{x}\bigl\{\mbox{$N_k^{[x_0,\infty)} \geq
m$}\bigr\}(1-\varepsilon).\nonumber
\end{eqnarray}
On the one hand, we have, for any $y>c$,
\[
\mathbb{P}_{x}\{ \mbox{$f(v) > y$ for every $v$ of height
$1$}\} = \sum_{b = 1}^{\infty} \nu(b)\P\{A_c(x)/x > y/x\}^{b}.
\]
Since the sequence $A_c(x) /x$ converges as $x\to\infty$ toward a
random variable which has no atom at $0$ (cf. Section
\ref{sectionprelim01}), the previous equality implies
\[
\lim_{x\to\infty}\mathbb{P}_{x}\{ \mbox{$f(v) > y$ for every
$v$ of height $1$}\} = 1.
\]
Using again the stochastic monotonicity of $F$ w.r.t. its starting
point, it follows by induction that, for any fixed $k$,
%
\begin{equation}\label{inter1}
\lim_{x\to\infty}\mathbb{P}_{x}\{ \mbox{$f(v) > x_0$ for every
$v\in\T$ s.t. $|v|=k$}\} = 1.
\end{equation}
On the other hand, the tree $\T$ grows exponentially so that, for
any $m$,
%
\begin{equation}\label{inter2}
\lim_{k\to\infty}\mathbb{P}\bigl\{ \#\{v\in\T,  |v|= k\} \geq m\bigr\} = 1.
\end{equation}
Combining (\ref{inter1}) and (\ref{inter2}), we deduce that, for any
$m$, we can find $k$ and $x$ large enough such that
%
\begin{equation}
\mathbb{P}_{x}\bigl\{\mbox{$N_k^{[x_0,\infty)} \geq m$}\bigr\} \geq
1-\varepsilon,
\end{equation}
which yields, using (\ref{inter0}),
\[
\mathbb{P}_{x}\{\mbox{$F$ survives}\}\geq(1-\varepsilon)^2.
\]
\upqed\end{pf}

\subsection{\texorpdfstring{Proof of transience when $b\mu(c)>1$}{Proof of transience when b mu(c)>1}}\label{sectiontransience}
Let $(Z_n)_{n\geq0}$ be a Markov chain on $[c,\infty)$ with
transition probabilities given by (\ref{transitionProbaZ}) and
denote by $\P_{x}$ the probability under which $Z$ starts from $Z_0
=x$. Let $T\in\mathcal{T}$ and fix $v\in T$. It follows from the
definition of the branching Markov chain $F$ that
\[
\mathbb{P}_{x}\{f(v) \in E  |  \T=T\} =\P_{x}\bigl\{Z_{|v|}\in E \bigr\}.
\]
Let us for the time being admit that, for some $x_0 > c$, we have
%
\begin{equation}\label{expodecayZ}
\liminf_{n\rightarrow\infty} \P_{x_0}\{Z_n \geq
x_0\}^{1/n} \geq\mu(c).
\end{equation}
Recalling that $N_k^{[x_0,\infty)}$ denotes the number of particles
of $F$ located above level $x_0$ at time $k$, we find, when $\mu(c)b
> 1$, that for $k_0$ large enough,
\begin{eqnarray*}
\mathbb{E}_{x_0}\bigl[N_{k_0}^{[x_0,\infty)}\bigr]&=&\mathbb{E}_{x_0}
\biggl[\sum
_{|v| = k_0} \un_{\{f(v)\geq x_0\}}\biggr]\nonumber\\
&=& \mathbb{E}[\# \{v\in\T,  |v|=k_0\}]\P_{x_0}\{Z_{k_0} \geq x_0\}
\nonumber
\\[-8pt]
\\[-8pt]
\nonumber
&=& (b\P_{x_0}\{Z_{k_0}\geq x_0\}^{1/k_0})^{k_0} \\
&\geq& 2.\nonumber
\end{eqnarray*}
Just as in the proof of Proposition \ref{prop0-1}, making use of the
branching property of~$F$ and keeping only the particles located
above $x_0$ at times $k_0 n$, $n\geq0$, it follows by induction
that, under $\mathbb{P}_{x_0}$, the process $(N_{k_0
n}^{[x_0,\infty)})_{n\geq0}$ stochastically dominates a classical
Galton--Watson process with reproduction law
$N_{k_0}^{[x_0,\infty)}$. Since
$\mathbb{E}_{x_0}[N_{k_0}^{[x_0,\infty)}] \geq2$, this
Galton--Watson process has probability $\alpha>0$ of non-extinction,
which implies
\[
\mathbb{P}_{x_0}\{\mbox{$F$ survives}\} \geq\alpha.
\]
We conclude using (\ref{limitFsurvive1}) and Proposition
\ref{prop0-1} that
\[
\mathbb{P}\{\mbox{$X$ visits each vertex of $\T$ finitely many
times}\}=1.
\]
It remains to prove (\ref{expodecayZ}) which is a consequence of

\begin{duge}\label{keylemma} Let $(S(x),x \in\R)$ be a collection of
real-valued random variables. Assume that the following hold:
\begin{longlist}[(a)]
\item[(a)] For any $x < y$, the random variable $x+ S(x)$ is
stochastically dominated by $y + S(y)$.
\item[(b)] $S(x)$ converges in law, as $x$ tends to $+\infty$, toward a
random variable~$S(\infty)$ whose law is absolutely continuous w.r.t.
the Lebesgue
measure and $\P\{S(\infty) > 0\} >0$.
\item[(c)] The Laplace transform $\phi(\lambda) \defeq\mathbf
{E}[e^{\lambda S(\infty)}]$ reaches its minimum at some point $\rho>
0$ which belongs to the
nonempty interior of its definition domain $\mathcal{D} \defeq\{
\lambda\in\R, \phi(\lambda)<\infty\}$.
\end{longlist}
Let $Y = (Y_n, n\geq0)$ denote a real-valued Markov chain with
transition kernel $\P\{Y_{n+1} \in E  |  Y_n = y\} = \P\{ S(y) + y
\in E\}$ and let $\tau_x$ be the first time $Y$ enters the interval
$(-\infty,x)$. Denoting by $\P_x$ the law of $Y$ starting from $x$,
we have, for all $x$ large enough,
%
\begin{equation}\label{mainlemmres}
\lim_{n\to\infty} \P_x\{\tau_x>n\}^{1/n}\geq\phi(\rho).
\end{equation}
\end{duge}

We apply the lemma to the Markov chain $Y$ defined by
\[
Y_n \defeq\log Z_{n}.
\]
According to (\ref{transitionProbaZ}), we have
\[
\P\{Y_{n+1} \in E  |  Y_n = y\} = \P\{ S(y) + y \in E\}
\]
with $S(y) \defeq\log m_c(\exp(y))$ and $S(\infty) \defeq\log
m_c(\infty)$ where $m_c$ is the martingale of Section
\ref{sectionprelim01}. On the one hand, assumption (a) holds since
$A_c(x) \leq A_c(y)$ for all $x\leq y$. On the other hand, the
results of Davis and Volkov \cite{DavisVolkov02,DavisVolkov04}
recalled in Section~\ref{sectionprelim01} imply that assumptions
(b),(c) also hold and
\[
\inf_{\lambda\in\R}\E\bigl[e^{\lambda S(\infty)}\bigr] = \mu(c).
\]
Thus, we conclude that, for $x_0$ large enough,
\[
\liminf_{n\rightarrow\infty} \P_{x_0}\{Z_n \geq
x_0\}^{1/n} \geq\lim_{n\rightarrow\infty} \P_{\log
x_0}\Bigl\{\min_{1\leq i\leq n}Y_i \geq\log x_0\Bigr\}^{1/n} \geq
\mu(c).\vadjust{\goodbreak}
\]

\begin{pf*}{Proof of Lemma \protect\ref{keylemma}}
Assumption (a) implies that for $x < y$, the Markov chain $Y$ under
$\P_x$ is stochastically dominated by $Y$ under $\P_y$. Thus, using
the Markov property, we get that, for any $n,m$,
\[
\P_x\{\tau_x > n+m\} \geq\P_x\{\tau_x>n\}\P_x\{\tau_x>m\}.
\]
The superadditivity of the sequence $\log\P_x\{\tau_x>n\}$ now
implies that the limit in~(\ref{mainlemmres}) exists. It remains to
prove the lower bound for $x$ large enough.

Set $g_x(t) \defeq\P\{S(x) > t\}$ and $g(t) \defeq\P\{S(\infty) >
t\}$. In view of assumption (b), as $x$ goes to $+\infty$, $g_x$
converges uniformly toward $g$. Define
\[
\hat{g}_x(t) \defeq\inf_{y\geq x}g_y(t).
\]
For each $x$, the function $\hat{g}_x$ is c\`{a}dl\`{a}g, non-increasing,
with $\lim_{t\to-\infty}\hat{g}_x(t) = 1$ and $\lim_{t\to
+\infty}\hat{g}_x(t) = 0$. Thus, for each $x$, we can consider a
random variable~$\hat{S}(x)$ such that $\P\{\hat{S}(x) > t\} =
\hat{g}_x(t)$. By construction, the sequence of random variables
$\hat{S}(x)$ is stochastically monotonic and converges in law toward
the random variable $S(\infty)$. Let $\hat{Y}^x$ denote a random
walk with step~$\hat{S}(x)$, that is, $\hat{Y}^{x}_{n+1} -
\hat{Y}^{x}_n \stackrel{\mathrm{law}}{=} \hat{S}(x)$.

By construction, the random variable $\hat{S}(x)$ is stochastically
dominated by~$S(y)$ for any $y\geq x$. Combining this fact and the
stochastic monotonicity of the Markov chain $Y$ w.r.t. its starting
point, it follows by induction that the random walk $\hat{Y}^x$
started from $x$ and killed when it enters the interval
$(-\infty,x)$ is stochastically dominated by $Y$ under $\P_{x}$. In
particular, denoting by $\hat{\tau}^x_0$ the first time $\hat{Y}^x$
enters the interval $(-\infty,0)$, it follows that $\hat{\tau}^x_0$
under~$\P_0$ (i.e., the walk $\hat{Y}^x$ started from $0$) is
stochastically dominated by $\tau_x$ under~$\P_x$. Hence,
%
\begin{equation}\label{mainlemmE1}
\lim_{n\to\infty} \P_x\{\tau_x>n\}^{1/n}\geq\liminf_{n\to\infty}
\P_0\{\hat{\tau}^{x}_0 >n\}^{1/n}.
\end{equation}
Let $\hat{\phi}_x(\lambda) \defeq\mathbf{E}[e^{\lambda
\hat{S}(x)}]$ with definition domain $\hat{\mathcal{D}}_x \defeq\{
\lambda\in\R, \hat{\phi}_x(\lambda)<\infty\}$. Since $\hat{S}(x)$
is stochastically dominated by $S(\infty)$, we have $\mathcal{D}
\cap[0,\infty) \subset\hat{\mathcal{D}}_x \cap[0,\infty)$.
According to assumption (c), we can choose $a>0$ such that $I_a
\defeq[\rho-a,\rho+a] \subset\mathcal{D}\cap[0,\infty)$. On
$I_a$, as $x$ goes to $+\infty$, the functions $\hat{\phi}_x$
converge uniformly toward~$\phi$. Making use of the strict convexity
of a Laplace transform, it follows that, for all $x$ large enough,
the function $\hat{\phi}_x$ verifies assumption (c), that is, $\hat
{\phi}_x$ reaches its minimum on $\hat{\mathcal{D}}_x$ at some
point $\rho_x \in I_a$. Moveover, we have
%
\begin{equation}\label{mainlemmE1A}
\lim_{x\to\infty} \hat{\phi}_x(\rho_x) = \phi(\rho).
\end{equation}
Applying now Theorem 1 of \cite{BigginsLubachevskyShwartzWeiss91} to
the random walk $\hat{Y}^x$ with step distribution~$\hat{S}(x)$
gives
%
\begin{equation}\label{mainlemmE2}
\liminf_{n\to\infty} \P_0\{\hat{\tau}^{x}_0 >n\}^{1/n} =
\hat{\phi}_x(\rho_x).
\end{equation}
Combining (\ref{mainlemmE1}) and (\ref{mainlemmE2}), we get that
%
\begin{equation}\label{mainlemmE3}
\lim_{n\to\infty} \P_x\{\tau_x>n\}^{1/n} \geq\hat{\phi}_x(\rho_x).
\end{equation}
Assumption (b) also implies that, for some $\varepsilon,\eta>0$
small enough, there exists $x_0$ such that, for all $x\geq x_0$, we
have $\P\{S(x) > \varepsilon\} > \eta$, thus $\P\{Y_{n+1} >
\varepsilon+ Y_{n}   |   Y_n = x\} > \eta$. In particular, for
$x > y > x_0$, the event $\mathcal{E}(x,y) = \{ \mbox{$Y$ enters}\break
[x,\infty)\ \mbox{before entering $(-\infty,y)$} \}$ has a strictly
positive probability under $\P_y$. Therefore, using again the Markov
property and the stochastic monotonicity of $Y$ w.r.t. its starting
point, we get
\[
\P_{y}\{\tau_y > n\} \geq\P_y\{\mathcal{E}(x,y)\}\P_{x}\{\tau_x >
n\}
\]
which yields
%
\begin{equation}\label{mainlemmE4}
\lim_{n\to\infty} \P_y\{\tau_y>n\}^{1/n} \geq\lim_{n\to\infty}
\P_x\{\tau_x>n\}^{1/n}.
\end{equation}
Combining (\ref{mainlemmE1A}), (\ref{mainlemmE3}) and
(\ref{mainlemmE4}), we conclude that, for $y \geq x_0$,
%
\begin{equation}
\lim_{n\to\infty} \P_y\{\tau_y>n\}^{1/n} \geq\lim_{x\to+\infty}
\hat{\phi}_x(\rho_x) = \phi(\rho).
\end{equation}
\upqed\end{pf*}

\begin{rem}
$\!\!\!$Suppose that the $\operatorname{VRJP}(c)$ is recurrent on~$\T$. Recall that~$\xi(t)$
denotes the time where the local time of the walk at the
origin reaches $t-c$. We can express $\xi(t)$ in terms of the
branching Markov chain $F$ and we get, using that $\E_{t}[Z_n] = t$
for all $n$,
%
\begin{equation}\label{nulrec}
\qquad\mathbb{E}[\xi(t)] = \mathbb{E}_{t}\biggl[ \sum_{v\in\T}\bigl(f(v) -
c\bigr)\biggr] = \sum_{n=0}^{\infty} b^n\E_{t}[Z_n -c] =
\sum_{n=0}^{\infty} b^n(t -c) = \infty
\end{equation}
for any $t>c$. In particular, denoting by $\zeta_o$ the first time
the walk returns to the root of the tree, it easily follows from
(\ref{nulrec}), by conditioning on the time the walk makes its first
jump and applying the restriction principle, that any recurrent VRJP
on $\T$ is ``null'' recurrent in the sense that $\mathbb{E}[\zeta_o] =
\infty$.
\end{rem}

\subsection{\texorpdfstring{The critical case $b \mu(c) = 1$}{The critical case b mu(c)=1}}\label{sectioncritique}

The following proposition directly implies that the $\operatorname{VRJP}(c)$ on a
Galton--Watson tree is recurrent in the critical case \mbox{$b\mu(c) =1$}
since we already know that recurrence occurs when $b \mu(c) < 1$.

\begin{prop}\label{propcritique}
Assume that the $\operatorname{VRJP}(c)$ is transient on some Gal\-ton--Watson tree $\T
$ without leaves and with mean $b>1$. Then, there exists a~Gal\-ton--Watson tree $\tilde{\T}$
(with leaves) with mean $1 < \tilde{b} < b$ such that the $\operatorname{VRJP}(c)$ on
$\tilde{\T}$ is also transient on the event that $\tilde{\T}$ is infinite.
\end{prop}

The proof of Proposition \ref{propcritique} uses again the
characterization of transience in terms of the positive probability
of survival of the associated branching Markov chain $F$. Roughly
speaking, we show that, conditionally on survival, the number of
particles of $F$ not located at $c$ grows exponentially with time.
This implies\vadjust{\goodbreak} that the branching Markov chain on a small percolation
of the original tree still survives with positive probability. Hence
the VRJP on this percolated tree is also transient.

In the following, we assume as before that the Galton--Watson tree
$\T$ with reproduction law $\nu$ has no leaves and has mean $b>1$ so
that it is infinite and grows exponentially. Recall the definition
of the branching Markov chain $F = (f(v),  v\in\T)$ constructed in
Section \ref{sectionBMC}. We denote by $(\mathcal{F}_n)$ the
natural filtration of $F$:
\[
\mathcal{F}_n \defeq\sigma\bigl(\T_n , (f(v),  v\in\T_n)\bigr).
\]

\begin{duge}\label{lemmecritique} Recall the definition of $N^I_n$
given in (\ref{defNI}). Let $\mathcal{E}(x,k)$ be the event
\[
\mathcal{E}(x,k)\defeq\bigl\{ \mbox{There exist infinitely many $n$ such
that $N_n^{[x,\infty)} \geq k$ } \bigr\}.
\]
For any starting point $x_0 > c$, we have
\[
\mathcal{E}(x_0,2)=\{F\mbox{ survives}\}\qquad
\mbox{$\mathbb{P}_{x_0}$-a.s.}
\]
\end{duge}

\begin{pf}The inclusion $\mathcal{E}(x_0,2)\subset\{F\mbox{
survives}\}$ is trivial. Let $\varepsilon>0$ and set, for $k\leq n$,
\[
B_{k,n} \defeq\mathbb{E}_{x_0}\bigl[N^{(c,\infty)}_n \un_{\{
N_k^{[c+\varepsilon,\infty)} = 0,N_{k+1}^{[c+\varepsilon,\infty)} =
0,\ldots,N_{n-1}^{[c+\varepsilon,\infty)} = 0 \} } \bigr].
\]
Recall that each particle $v$ of $F$ evolves independently and gives
birth to a~random number $B$ (with mean $b$) of children. Moreover,
conditionally on~$\mathcal{F}_n$, the positions\vspace*{-1pt}
$f(\fils{v}{1}),\ldots,f(\fils{v}{B})$ of the children of a particle
$v$ at time~$n$ (i.e., $|v|=n$) are i.i.d. and distributed as
$A_c(f(v))$. Thus, in view of~(\ref{tes}), it follows that
\[
\mathbb{E}\bigl[ N^{(c,\infty)}_{n+1} |  \mathcal{F}_n\bigr] \leq
b(1-e^{-c\varepsilon})N^{(c,\infty)}_{n} \qquad \mbox{on the event
$\bigl\{ N^{(c+\varepsilon,\infty)}_n = 0\bigr\}$.}
\]
Choosing $\varepsilon$ small enough such that
$b(1-e^{-c\varepsilon}) < 1/2$, we get
\begin{eqnarray*}
B_{k,n+1} &=& \mathbb{E}_{x_0}\bigl[ \mathbb{E}\bigl[
N^{(c,\infty)}_{n+1} |  \mathcal{F}_n\bigr]\un_{\{
N_k^{[c+\varepsilon,\infty)} =
0,N_{k+1}^{[c+\varepsilon,\infty)} = 0,\ldots,N_{n}^{[c+\varepsilon
,\infty)} = 0 \}} \bigr]\\
&\leq& \tfrac{1}{2}\mathbb{E}_{x_0}\bigl[ N^{(c,\infty)}_{n} \un
_{\{
N_k^{[c+\varepsilon,\infty)} = 0,N_{k+1}^{[c+\varepsilon,\infty)} =
0,\ldots,N_{n-1}^{[c+\varepsilon,\infty)} = 0 \}} \bigr]\\
&\leq& \tfrac{1}{2}B_{k,n},
\end{eqnarray*}
which yields
\begin{eqnarray*}
\mathbb{E}_{x_0}\Biggl[\Biggl( \sum_{n=k}^{\infty} N_n^{(c,\infty)}
\Biggr)\un_{\{N_{i}^{[c+\varepsilon,\infty)} = 0\ \mathrm{for}\ \mathrm{all}\ i\geq k\}} \Biggr] \leq\sum_{n=k}^{\infty} B_{k,n} <
\infty.
\end{eqnarray*}
Therefore, $F$ dies out $\mathbb{P}_{x_0}$-a.s. on the event
$\{\mbox{$N_{i}^{[c+\varepsilon,\infty)} = 0$ for all $i\geq
k$}\}$. Taking the limit as $k$ goes to infinity, we obtain
%
\begin{equation}\label{firstinclu}
\mathcal{E}(c+\varepsilon,1)\supset\{F\mbox{ survives}\}\qquad
\mbox{$\mathbb{P}_{x_0}$-a.s.}\vadjust{\goodbreak}
\end{equation}
Let now $U_n \defeq\un_{\{N^{[x_0,\infty)}_{n}\geq2 \} }$. Using
the stochastic monotonicity of Proposition \ref{propmonotonyF} and
the fact that $\nu[2,\infty) > 0$ (since $b>1$) and
(\ref{densitepositive}), we find that
\begin{eqnarray}\label{secondinclu}
\mathbb{E}[U_{n+1}  |   \mathcal{F}_n]     & \geq&
\mathbb{E}\bigl[U_{n+1} \un_{\{N^{[c+\varepsilon,\infty)}_{n}\geq1\} }
  |   \mathcal{F}_n\bigr]\nonumber\\
    &\geq&     \un_{\{N^{[c+\varepsilon,\infty
)}_{n}\geq1\}} \mathbb{E}_{c+\varepsilon}[U_{1}]\nonumber\\
    &\geq&
\un_{\{N^{[c+\varepsilon,\infty)}_{n}\geq1\}}
\mathbb{P}_{c+\varepsilon}\left\{\matrix{
\mbox{\scriptsize{the initial particle $o$ has at least two
children}}\vspace*{2pt}\cr
\mbox{\scriptsize{with $f(\fils{o}{1}) \geq x_0$ and $f(\fils{o}{2})
\geq x_0$}}}
\right\}\\
   &=&     \un_{\{N^{[c+\varepsilon,\infty
)}_{n}\geq
1\}} \nu[2,\infty) \P\{A_c(c+\varepsilon) > x_0\}^2\nonumber\\
    &=&     C \un_{\{N^{[c+\varepsilon,\infty)}_{n}\geq1\}}\nonumber
\end{eqnarray}
for some constant $C >0$. Combining (\ref{firstinclu}) and
(\ref{secondinclu}), we get
\[
\sum_{n=1}^{\infty}\mathbb{E}[U_{n+1}  |   \mathcal{F}_n] =
\infty\qquad
\mbox{on the event $\{F\mbox{ survives}\}$.}
\]
A direct application of the conditional Borel--Cantelli Lemma
(cf. \cite{Freedman73}) yields
\[
\sum_{n=1}^{\infty} U_n = \infty\qquad \mbox{on the event
$\{F\mbox{ survives}\}$}
\]
which exactly means that $\mathcal{E}(x_0,2)\supset\{F\mbox{
survives}\}$.
\end{pf}

\begin{pf*}{Proof of Proposition \protect\ref{propcritique}}
Assume that the $\operatorname{VRJP}(c)$ on the Galton--Watson tree $\T$ with
reproduction law $\nu$ is transient. According to Proposition~\ref{prop0-1}
and~(\ref{limitFsurvive1}), we have
\[
\lim_{x\to\infty}\mathbb{P}_x\{F \mbox{ survives}\} = 1.
\]
Define the (possibly infinite) $\mathcal{F}_n$-stopping time
\[
\sigma_x\defeq\inf\bigl\{k\geq1, N_k^{[x,\infty)} \geq2\bigr\}.
\]
Using the result of the previous lemma, we get
%
\begin{equation}\label{lasteq1}
\lim_{x\to\infty}\lim_{\gamma\to\infty}\mathbb{P}_x\{ \sigma
_x\le
\gamma\} = \lim_{x\to\infty}\mathbb{P}_x\{F \mbox{ survives}\} = 1.
\end{equation}
Let now $\tilde{\T}$ be the tree obtained from $\T$ by removing
independently each vertex (and its descendants) with probability
$\eta>0$. The tree $\tilde{\T}$ is again a Galton--Watson tree with
mean $\tilde{b}=b(1-\eta)<b$. We denote by $\tilde{F}$ the
restriction of $F$ to $\tilde{\T}$,
\[
\tilde{F} \defeq\bigl(f(v), v\in\tilde{\T}\bigr).
\]
The restriction principle states that $\tilde{F}$ is the branching
Markov chain associated with the $\operatorname{VRJP}(c)$\vadjust{\goodbreak} on $\tilde{\T}$. Let
$\tilde{M}$ be the number of particles in $\tilde{F}$ located above
$x$ at time $\sigma_x$,
\[
\tilde{M} \defeq\#\{v\in\tilde{\T}, |v|=\sigma_x, f(v)>x\}
\]
(with the convention $\tilde{M}=0$ when $\sigma_x = \infty$). We
have
%
\begin{eqnarray}\label{lasteq2}
\mathbb{E}_{x}[\tilde{M}] &\geq& \mathbb{E}_{x}\bigl[\tilde
{M}\un
_{\{\sigma_x\le\gamma\}}\un_{\{\T_\gamma= \tilde{\T}_\gamma\}
}\bigr]\nonumber\\
&\geq& 2\mathbb{P}_{x}\{ \sigma_x\le\gamma\mbox{ and
}\T
_\gamma= \tilde{\T}_\gamma\}
\nonumber
\\[-8pt]
\\[-8pt]
\nonumber
&\geq& 2(\mathbb{P}_x\{ \sigma_x \leq\gamma\} +
\mathbb
{Q}_{\nu}\{\T_{\gamma} = \tilde{\T}_{\gamma}\} - 1)\\
&\geq& 2\bigl(\mathbb{P}_x\{ \sigma_x\le\gamma\} +
\mathbb{Q}_{\nu}\{\#\T_{\gamma} \leq b^{2\gamma}
\}(1-\eta)^{b^{2\gamma}} -1 \bigr).\nonumber
\end{eqnarray}
Recalling that the distribution of offsprings $\nu$ has mean $b$, we
get
%
\begin{equation}\label{lasteq3}
\lim_{\gamma\to\infty}\mathbb{Q}_{\nu}\{\#\T_{\gamma} \leq
b^{2\gamma} \} = 1.
\end{equation}
Combining (\ref{lasteq1}), (\ref{lasteq2}) and (\ref{lasteq3}), we
can choose $x,\gamma$ large enough and $\eta>0$ small enough such
that
\[
\mathbb{E}_{x}[\tilde{M}] > 1.
\]
Finally, using again the branching structure of $\tilde{F}$ and the
stochastic monotonicity of the process w.r.t. the position of the
initial particle, it follows by induction that the random variable
$\#\{v\in\tilde{\T}, f(v)>x\}$ under $\mathbb{P}_x$ is
stochastically larger than the total progeny of a Galton--Watson
process with reproduction law $\tilde{M}$. Since
\mbox{$\mathbb{E}[\tilde{M}]>1$}, this process is supercritical, hence
\[
\mathbb{P}_{x}\{\mbox{$\tilde{F}$ survives}\} \geq
\mathbb{P}_{x}\bigl\{\#\{v\in\tilde{\T}, f(v)>x\} = \infty\bigr\} > 0
\]
which in turn implies that the $\operatorname{VRJP}(c)$ on the percolated tree
$\tilde{\T}$ is transient on the event that $\tilde{\T}$ is
infinite.
\end{pf*}

%


\printaddresses

\end{document}